\documentclass[a4paper,11pt]{article}
\usepackage{amsmath}
\usepackage{amssymb}
\usepackage{amsthm}
\usepackage{hyperref}
\hypersetup{
    colorlinks = false
}
\usepackage{color}
\newtheorem{theorem}{Theorem}[section]
\newtheorem{corollary}[theorem]{Corollary}
\newtheorem{lemma}[theorem]{Lemma}

\newtheorem{proposition}[theorem]{Proposition}
\theoremstyle{remark}

\newtheorem*{lemma*}{{\bf{Lemma}}}
\newtheorem*{remark}{{\bf{Remark}}}
\newcommand{\A}{{\mathcal{A}}}
\newcommand{\B}{{\mathcal{B}}}
\newcommand{\N}{{\mathbb{N}}}
\newcommand{\Z}{{\mathbb{Z}}}
\newcommand{\F}{{\mathbb{F}}}

\newif\ifshowsection
\showsectionfalse

\begin{document}

\author{Rainer Dietmann, Christian Elsholtz, 
Imre Ruzsa}
\title{Sieving with square conditions and applications 
to Hilbert cubes in arithmetic sets}  
\maketitle

\begin{abstract}
The purpose of this paper is twofold:
1) Applications of Gallagher's larger sieve modulo prime squares do not work. In some relevant cases we can transform the residue class  information modulo $p^2$ to more suitable residue information modulo $p$, so that we can successfully apply the sieve.

2) The applications to Hilbert cubes are of interest in their own right:
We study the maximal dimension of Hilbert cubes in various multiplicatively defined sets. For the squareful numbers in $[1,N]$  we achieve an upper bound of the dimension of $d=O(\log N)$. The same upper bounds follow for multiplicative semigroups of integers defined by a positive proportion of the primes, and the set of integers representable by an irreducible positive definite binary quadratic form. Eventually, making use of the sun flower lemma we give an improvement on the maximal dimension $d$ of subset sums in the set of pure powers in $[1,N]$.
\end{abstract}
\section{Introduction}
\subsection{Aims} There is an extensive literature\footnote{The authors acknowledge that today, on March 17, 2026, a paper with overlapping content has been posted by Croot, Mao and Yip on ArXiV, see \texttt{https://arxiv.org/pdf/2603.14654}, which prompted us to upload our paper.}
on the subject of sumsets
${\cal A+B \subseteq  S}:=\{a+b: a \in \A, b \in \B\}$
contained
in sets of arithmetic interest, such as the set of squares, primes, squareful
numbers, square-free numbers etc.
Also, a number of papers studied iterated sumsets,
e.g.~the dimension $d$ of Hilbert cubes
$\mathcal{H}(a_0;a_1, \ldots , a_d)=a_0 + \{0,a_1\}+\cdots + \{0,a_d\}\subseteq {\cal S}$
or subset sum cubes, i.e.~Hilbert cubes with $a_0=0$.
For many of these papers it has been crucial
to study the question modulo primes, and use this local information via a
sieve method to find bounds on the counting function
of the additive components ${\cal A,B}$ or the maximal dimension $d$.

In this paper we make progress by studying the problem modulo composite
numbers, and adapt required results from additive combinatorics and
sieve theory accordingly.

We hope that this paper is not only of interest because of its applications to
arithmetic sets, but also because of its methodological progress on sieves.

\subsection{Notation}
If $q$ is a positive integer, then we write $\Z_q$ for
$\Z/q\Z$. In particular, $\Z_p$ is the finite field $\F_p$
if $p$ is a prime.

For a set $\A$ of positive integers, we write
$\A(N)=\sum_{n \le N \atop n \in \A} 1$.

Let $\Sigma(\A)=\sum_{a \in \A} a$. 

As usual, we write $\lceil x \rceil$ for the smallest integer at least
as big as $x$.

\subsection{Old and new tools from additive combinatorics}

The following well known result is one of our main tools.
\begin{lemma}[Olson \cite{Olson:1968}, Theorem 1]
\label{olson}
Let $p$ be a prime, $\B \subset \Z/p\Z$ with $|\B|>2\sqrt{p}$ and
$a \in \Z/p\Z$. Then there exists a non-empty $\A \subset \B$ such that
$\sum(\A) \equiv a \pmod p$.
\end{lemma}

The following result does not seem to exist in the literature but is key
to our application.

\begin{theorem}{\label{Ruzsa}}
Let $p$ be a prime and $q=pm$, for some integer $m>1$,
$\B\subset {\mathbb {Z}}_q$, $\left|\B \right|=n$. 
Suppose that the reduction of $\B$ modulo $p$ has more than
$4 \lceil 2 \sqrt{p} \rceil$ elements,
and assume that at least one element of $\B$ is not a multiple of $m$.
Then there is a subset $\A\subset \B$ with 
$\Sigma(\A)\equiv 0 \bmod p$ and 
$\Sigma(\A) \not\equiv 0 \bmod q$.
\end{theorem}

\begin{corollary}
\label{schwarzwald}
Let $p$ be a prime and let $q=p^\ell$ where $\ell>1$. Further, let
$\B \subset \Z_q$ with $|\B \mod p| \ge 5 \lceil 2 \sqrt{p} \rceil+2$.
Moreover, let $a_0 \in \Z_q$. Then there exists $\A \subset \B$ with
\begin{equation}
\label{wind}
  a_0+\sum (\A) \equiv 0 \pmod p
\end{equation}
and
\begin{equation}
\label{kuehl}
  a_0+\sum (\A) \not \equiv 0 \pmod q.
\end{equation}
\end{corollary}

\subsection{Main results: application to problems on Hilbert cubes
in multiplicatively defined sets}

The main results of this paper are new upper bounds on the dimension of Hilbert cubes or subset sums in certain multiplicatively defined sets.

Let us first recall what is known in the cases of squares and primes:
When $\mathcal{S}$ is the set of squares, then it is known that
$\min(A(N), B(N))=O(\log N)$, see \cite{Gyarmati:2001}.
Also, the maximal dimension of a Hilbert cube in the set of squares ${\cal S}\cap [1,N]$ is
$d=O(\log\log N)$, see \cite{DietmannandElsholtz2}, improving
\cite{DietmannandElsholtz,HegyvariandSarkozy:1999}.

Let ${\cal P}$ denote the set of all primes.
If ${\cal A+B \subseteq P}$, then ${\cal A}(N){\cal B}(N)=O(N)$,
\cite{Wirsing, PSS}.
Moreover, for all $\varepsilon >0$, there exists a $k$ such that
when ${\cal A}(N) \gg (\log N)^k$, then
${\cal B}(N) \leq N^{\frac{1}{2}+ \varepsilon}$, see \cite{Elsholtz:1999}.
The maximal dimension of Hilbert cubes in the set of primes, in the finite
interval $[1,N]$ is $d=O(\frac{\log N}{\log \log N})$, see
Woods \cite{Woods} and Elsholtz \cite{Elsholtz:2003},
improving Hegyvari and S\'{a}rk\"ozy \cite{HegyvariandSarkozy:1999}.

In this paper we study maximal dimensions in sets such as the squareful numbers, pure powers, or the value set of binary quadratic forms.

Let ${\cal W}=\{m \in \N: p\mid n \text{ implies } p^2\mid n\}$ be the set of squareful numbers, and suppose that
${\cal A}+{\cal B}\subset {\cal W}$. There do not seem to be upper bounds on $|{\cal A}|, |{\cal B}|$ in the literature, but an upper bound of the dimension of Hilbert cubes in ${\cal W}\cap [1,N]$ is $O((\log N)^2)$ (see Theorem 1.5 in \cite{DietmannandElsholtz2}).
For subset sums in the set of pure powers $V \cap [1,N]$,
the bound $d=O\left(\frac{(\log \log N)^3}{\log \log \log N}\right)$ (see Theorem 1.7 in \cite{DietmannandElsholtz2}) is known. We improve on these last two results in Corollary \ref{cor:margaretenbad} and Theorem \ref{thm:pure-powers}.

These results have the following applications to the largest
possible dimension $d$
of Hilbert cubes in certain multiplicatively defined sets.
Let us consider $r$-full numbers, where the squareful numbers ($r=2$) are the most important case.
\begin{theorem}[Hilbert cubes in semigroups, motivated by norm forms]\label{densitytheorem}
Let $r \ge 2$ be a fixed integer and
let ${\cal T}$ denote a set of primes and $\tau$ a positive constant  with
\[
\sum_{p \in {\cal T},\, p \leq y} \frac{\log p}{\sqrt{p}}
\geq (\tau+o(1)) y^{\frac{1}{2}}.\]
Let $ {\cal S_{T}}= \{n \in \N_0: p\mid n \text{ and } p \in {\cal T}
\text{ implies } p^r \mid n \}$.
Let $F_1(N)$ denote the largest dimension $d$ such that
there exist a non-negative integer $a_0$ and positive integers
$a_1, \ldots, a_d$ with
$\mathcal{H}(a_0; a_1, \ldots, a_d) \subset \mathcal{S}_{\mathcal{T}}
\cap [1,N]$.
Then the following holds:
\[ F_1(N) \ll_{\mathcal \tau} \log N.\]
\end{theorem}
\begin{remark}
Applying partial summation to the prime number theorem 
$\pi(y)\sim \frac{y}{\log y}$ it follows that
$\sum_{p \leq y} \frac{\log p}{\sqrt{p}}\sim 2\sqrt{y}$ holds.
Hence the condition
\[
\sum_{p \in {\cal T},\  p \leq y} \frac{\log p}{\sqrt{p}}
\geq (2\tau+o(1)) y^{\frac{1}{2}}\]
is satisfied, for example, for sets of primes with a lower
density $\tau$ (relative
to the primes) with $0 < \tau \leq 1$, in particular for primes in arithmetic
progressions $ak+b$, with $\gcd(a,b)=1$.

\end{remark}
This theorem has a number of corollaries for concrete sets of arithmetic
interest.

As a corollary we improve a result on Hilbert cubes in the set of 
 squareful numbers; it follows immediately from
Theorem \ref{densitytheorem} by choosing $\cal{T}$ to be the set of all
primes.

\begin{corollary}[Hilbert cubes in squareful numbers]
\label{cor:margaretenbad}
Let $\mathcal{W}=\{n \in \N_0: p\mid n \text{ implies } p^2 \mid n \}$
denote the set of squareful integers.
Let $F_2(N)$ denote the largest dimension $d$ such that
there exist a non-negative integer $a_0$ and positive integers
$a_1, \ldots, a_d$ with
$\mathcal{H}(a_0; a_1, \ldots, a_d) \subset \mathcal{W}
\cap [1,N]$. Then the following holds:
\[ F_2(N) \ll (\log N).\]
\end{corollary}
This improves on the bounds
\[ O(\log N)^3(\log \log N)^{-\frac{1}{2}})\]
proved by Gyarmati, S\'ark\"ozy and Stewart 
\cite{GyarmatiSarkozyandStewart:2003} in the special case of subset sums
and
\[ F_2(N) \ll (\log N)^2\]
proved by Dietmann and Elsholtz \cite{DietmannandElsholtz2}.
Gyarmati, S\'ark\"ozy and Stewart 
first provided local information, namely
that such a sequence can have 
integers in at most $p\sqrt{\log p}$ 
distinct residue classes modulo $p^2$ and converted the ``local knowledge''
modulo $p^2$  to 
a ``global'' upper bound. However,
this local bound is not sharp enough to enable one to use a sieve due to 
Gallagher, as we do below.
As for a lower bound, Gyarmati, S\'ark\"ozy and Stewart 
\cite{GyarmatiSarkozyandStewart:2003}
explicitly 
construct a set which shows 
that $F_2(N)\gg (\log N)^{\frac{1}{2}}$ holds.

The method by Dietmann and Elsholtz \cite{DietmannandElsholtz2} 
first used that the length of arithmetic progressions in the
squareful numbers $\mathcal{W} \cap [1,N]$
can be bounded by $O(\log N)$, 
and then an ``exponential growth'' of the size of Hilbert cubes in sets without long
arithmetic progressions.

We take the opportunity to also improve on the dimension of subset sum cubes in pure powers.
In this case the improvement comes from a recently improved result on sunflowers.
The result is 
\begin{theorem}[Subsetsums in pure powers]\label{thm:pure-powers}
Let $\mathcal{V}=\{a^n: a,n \in \N, a\geq 1, n\geq 2\}$
denote the set of pure powers.

Let $F_3(N)$ denote the largest dimension $d$ 
of a Hilbert cube
$\mathcal{H}(0; a_1, \ldots, a_d) \subset \mathcal{V}
\cap [1,N]$ with distinct $a_1, \ldots ,a_d$. Then the following holds:
\[ F_3(N) \ll (\log \log N)^2.\]
\end{theorem}

Another application of Theorem 
\ref{densitytheorem} is to Hilbert cubes in the value set of binary quadratic
forms, or more generally norm forms. Here it is known that there is an
`inert'
set of primes with positive asymptotic density (see Lemma 3.1 in
\cite{ElsholtzandFrei}).

\begin{corollary}[Hilbert cubes in binary quadratic forms]\label{cor:sumsof2squares}
Let $a, b, c \in \Z$ such that
$a x^2 + bxy + cy^2$ is an irreducible positive definite binary quadratic form,
and let ${\cal S}_{a,b,c}=\{ax^2+bxy+cy^2: x,y \in \Z\}$.
Let $F_{a,b,c}(N)$ denote the largest dimension $d$ such that
there exist a non-negative integer $a_0$ and positive integers
$a_1, \ldots, a_d$ with
$\mathcal{H}(a_0; a_1, \ldots, a_d) \subset \mathcal{S}_{a,b,c} \cap [1,N]$.
Then the following holds:
\[ F_{a,b,c}(N) \ll_{a,b,c} \log N.\]
\end{corollary}

This covers in particular the case
of ``sums of two squares'' ${\cal S}_2=\{x^2+y^2: x,y \in \N\}$. This set is 
obviously additively decomposable into 
${\cal S}_2 = \{x^2: x\in \N \}+ \{y^2:y \in \N\}$,
and the Hilbert cubes in the squares in $[1,N]$ have a maximal dimension
$d=O\left( (\log \log N)\right)$. 
This however does not automatically imply any upper bound
on the maximal size of Hilbert cubes in $S_2$, as 
there could possibly be another entirely different 
decomposition ${\cal S}_2={\cal A}+{\cal B}$. In this situation, the 
fact that  $p\equiv 3 \mod 4, p\mid n \in {\cal S}_2$ implies that $p^2 \mid n$
immediately gives Corollary \ref{cor:sumsof2squares}.



\begin{theorem}[Hilbert cubes in multiplicative semigroups]{\label{2nddensitytheorem}}
Let $0< \tau\leq 1$ and  ${\cal T}$ denote a set of primes with
\[
\sum_{p \not\in {\cal T},\ p \leq y} \frac{\log p}{\sqrt{p}}
\gg_{\tau} y^{\frac{1}{2}}.\]
(In particular, a set of primes with density $\tau$ with $ 0< \tau <1$
relative to the primes satisfies this condition.)
Let $ {\cal S}= \{n \in \N: p\mid n \Rightarrow p \in {\cal T} \}$.

Let $F_4(N)$ denote the largest dimension $d$ such that
there exist a non-negative integer $a_0$ and positive integers
$a_1, \ldots, a_d$ with
$\mathcal{H}(a_0; a_1, \ldots, a_d) \subset \mathcal{S} \cap [1,N]$.
Then the following holds:
\[ F_4(N) \ll_{\mathcal{T}} \log N.\]
\end{theorem}


\subsection{Gallagher's larger sieve modulo powers}

Let us state Gallagher's larger sieve first.

\begin{lemma}[See \cite{Gallagher:1971}]{\label{lem:Gallagher}}
Let ${\cal{P'}}$ denote a set of primes or powers of primes such that 
${\cal{A}}$
lies modulo $p^i$ (for $p^i \in {\cal{P'}}$) in at most $\nu(p^i)$
residue classes. Then the following bound holds for all $N>1$,
provided the denominator is positive:
\[ \A(N) \leq \dfrac{ - \log N + \sum_{p^i \in {\cal{P'}}} \log p}{- \log N
+ \sum_{p^i \in {\cal{P'}}} \dfrac{ \log p}{\nu(p^i)} }.
\]
\end{lemma}

Gallagher's larger sieve is a very useful tool, when studying the size
of sequences which are in only a few residue classes modulo primes.
For the classical question of modelling the set of squares, 
by observing that it is modulo odd primes $p$ in $\frac{p+1}{2}$ residue
classes it gives the (almost) best possible result that there are $O(\sqrt{N})$ squares
$a^2\leq N$. The large sieve gives the same type of an upper bound.
For the set of squares this result is of course trivial, but the point is that
the same upper bound holds for any sequence which is modulo primes
in at most $\frac{p+1}{2}$ residue classes, and that such a general bound is best
possible in the case of the squares (or more generally for the values of quadratic polynomials).

Studying the related question of a sequence $\mathcal{S}$,
that avoids half of the residue classes modulo
prime squares $p^2$ one observes that the best one can hope for
by the large sieve is an upper bound 
of $\mathcal{S}(N) =O(N^{3/4})$ (compare Selberg \cite{Selberg:1977},
Erd\H{o}s and S\'{a}rk\"ozy \cite{ErdosandSarkozy}).
Unfortunately Gallagher's larger sieve is of no use in this situation.
In the expression $\sum_{p \leq y}\frac{\log p}{\nu(p^2)}$ one would rather need
an estimate such as $\nu(p^2)\ll p$ to make the denominator positive.

In this paper we show that Gallagher's larger sieve can also be
used modulo composite integers, if extra information on the
distribution is available. In this application,
not only the number of residue classes modulo
$p^2$ is needed, but also the distribution modulo $p$ is made use of.


\section{Proofs}
\subsection{Proof of Theorem \ref{Ruzsa} and Corollary \ref{schwarzwald}}


\begin{proposition}
\label{keyprop}
Assumptions:\\
Let $p$ be a prime and $k$ be a positive integer with the following
property: The subsums of every $k$-subset of $\Z_p$ contain
every element of $\Z_p$. Moreover, let $q=mp$ for some $m>1$, and
let $\B\subset \Z_q$ be a set such that 
$\left|\B \right|>4k$, the elements of $\B$
are all distinct modulo $p$ and they are not all multiples of $m=q/p$. 

Then there exists a subset $\A\subset \B$
with $\Sigma(\A)\equiv 0 \pmod{ p}$ and $\Sigma(\A)\not  \equiv 0
\pmod{ q}$.
\end{proposition}
\begin{proof}[Proof of Theorem \ref{Ruzsa}]
The theorem follows from Proposition \ref{keyprop} by
putting $k=\lceil 2 \sqrt{p} \rceil$, which is a permissible choice 
by Lemma \ref{olson}.
\end{proof}
We divide the proof of Proposition \ref{keyprop} into several key lemmata.
\begin{lemma}
\label{lemmaA}
Keep the assumptions from Proposition \ref{keyprop},
and in addition suppose that there is no subset $\A \subset \B$
with $\Sigma(\A) \equiv 0 \pmod p$ and $\Sigma(\A) \not \equiv 0 \pmod q$.
If $\mathcal{V}_1, \mathcal{V}_2 \subset
\B$ with $\Sigma (\mathcal{V}_1) \equiv \Sigma (\mathcal{V}_2) \pmod p$ but
$\Sigma (\mathcal{V}_1) \not \equiv \Sigma (\mathcal{V}_2) \pmod q$, then
$|\mathcal{V}_1 \cup \mathcal{V}_2| > |\B|-k>3k$.
\end{lemma}
\begin{proof}
Suppose that $\mathcal{V}_1, \mathcal{V}_2\subset \B$ are such that 
$\sum (\mathcal{V}_1) \equiv  \sum (\mathcal{V}_2) \pmod{ p}$ but
$\sum (\mathcal{V}_1) \not\equiv  \sum (\mathcal{V}_2) \pmod q$, 
and $\left|\mathcal{V}_1 \cup  \mathcal{V}_2 \right| \le |\B|-k$. Then 
$|\B \backslash (\mathcal{V}_1 \cup \mathcal{V}_2)| \ge k$
and by definition of $k$, we can find a subset $\mathcal{W} \subset
\B \backslash (\mathcal{V}_1 \cup \mathcal{V}_2)$ such that
$\Sigma (\mathcal{W}) \equiv
-\Sigma (\mathcal{V}_1) \pmod p$.
Then $\Sigma (\mathcal{V}_1 \cup \mathcal{W}) \equiv 0 \pmod p$, and as
$\Sigma (\mathcal{V}_1) \equiv \Sigma (\mathcal{V}_2) \pmod p$, also
$\Sigma (\mathcal{V}_2 \cup \mathcal{W}) \equiv 0 \pmod p$.
Now $\Sigma (\mathcal{V}_1) \not \equiv \Sigma (\mathcal{V}_2) \pmod q$, so
$\Sigma (\mathcal{V}_1 \cup \mathcal{W}) \not \equiv 0 \pmod q$
or $\Sigma (\mathcal{V}_2 \cup \mathcal{W}) \not \equiv 0 \pmod q$
contradicting our assumption that there is no subset $\A \subset \B$
with $\Sigma(\A) \equiv 0 \pmod p$ and $\Sigma(\A) \not \equiv 0 \pmod q$.
\end{proof}

\begin{lemma}
\label{lemmaB}
Keep the assumptions from Lemma \ref{lemmaA}. Let $x \in \Z_p$. Then
there exists $f(x) \in \Z_q$ such that whenever $\mathcal{V} \subset \B$
with
$|\mathcal{V}| \le k$ and $\Sigma (\mathcal{V}) \equiv x \pmod p$, then
$\Sigma (\mathcal{V}) \equiv f(x) \pmod q$. This defines a function $f: \Z_p \rightarrow \Z_q$.
\end{lemma}
\begin{proof}
This follows immediately from Lemma \ref{lemmaA}.
\end{proof}

\begin{lemma}
\label{lemmaC}
Keep the assumptions from Lemma \ref{lemmaA}. Define a map
$f: \Z_p \rightarrow \Z_q$ via Lemma \ref{lemmaB}. Then $f$
is a group homomorphism between the additive groups $\Z_p$ and
$\Z_q$.
\end{lemma}
\begin{proof}
Let $x,y\in \Z_p$, and choose a set $\mathcal{X}\subset \B$
with $\left|\mathcal{X} \right|\leq k$ and $\sum (\mathcal{X})
\equiv x \pmod p$
which is possible by definition of $k$.
Since $\left|\B \setminus  \mathcal{X} \right|>k$,
in a similar way we can find $\mathcal{Y} \subset \B \setminus \mathcal{X}$ 
with $\left|\mathcal{Y}
\right|\leq k$ and $\sum (\mathcal{Y}) \equiv y \pmod p$. 
Finally choose $\mathcal{W} \subset \mathcal{B}$ with $|\mathcal{W}|
\le k$ such that
$\sum (\mathcal{W}) \equiv x+y \pmod p$. Then
\[   \sum (\mathcal{X}\cup \mathcal{Y}) \equiv  \sum (\mathcal{W})
\equiv  x+y \pmod{ p},  \]
so by Lemma \ref{lemmaA} we have 
\[
  \sum (\mathcal{X}\cup \mathcal{Y}) \equiv  \sum (\mathcal{W}) \pmod q.
\]
Now $|\mathcal{X}| \le k$ and $\sum (\mathcal{X}) \equiv x \pmod p$, so
\[
  \sum (\mathcal{X}) \equiv f(x) \pmod q
\]
by Lemma \ref{lemmaB}. Analogously,
\[
  \sum (\mathcal{Y}) \equiv f(y) \pmod q,
\]
and
\[
  \sum (\mathcal{W}) \equiv f(x+y) \pmod q.
\]
We conclude that
\begin{align*}
   f(x+y) & \equiv  \sum (\mathcal{W}) \equiv
   \sum (\mathcal{X}\cup \mathcal{Y}) \equiv 
   \sum (\mathcal{X}) + \sum (\mathcal{Y})\\
   & \equiv  f(x) + f(y) \pmod q,\\
\end{align*}
as required.
\end{proof}

\begin{proof}[Proof of Proposition \ref{keyprop}]
Assume that the opposite holds, i.e.~that
there is no subset $\A \subset \B$
with $\Sigma(\A) \equiv 0 \pmod p$ and $\Sigma(\A) \not \equiv 0 \pmod q$.
Now take an arbitrary $b\in \B$. By considering the set $\{b\}$ we see that
$f(b')=b$, where $b'$ is the residue of $b$ modulo $p$
and $f:\Z_p \rightarrow \Z_q$ is the map introduced in Lemma \ref{lemmaB}. As we know
from Lemma \ref{lemmaC} that $f$ is a homomorphism, we have
\begin{align*}
   pb & \equiv pf(b') \equiv f(b')+\cdots+f(b') \equiv f(b'+\cdots+b')
   \equiv f(pb') \equiv f(0)\\ & \equiv 0 \pmod q,
\end{align*}
that is, $b$ is a multiple of $m$, in contradiction to the assumption that
not all elements of $B$ are multiples of $m$.
\end{proof}




\begin{proof}[Proof of Corollary \ref{schwarzwald}]
Choose $\A_1 \subset \B$ with $|\A_1 \mod p|=\lceil 2 \sqrt{p} \rceil+1$
and such that $\B \backslash \A_1$ still contains at least one element not
divisible by $m=p^{\ell-1}$. By Lemma \ref{olson}, there exists $\A_2 \subset \A_1$
such that
\[
  a_0+\sum (\A_2) \equiv 0 \pmod p.
\]
If $a_0+\sum (\A_2) \not \equiv 0 \pmod q$, then $\A=\A_2$ satisfies
\eqref{wind} and \eqref{kuehl} and we are done, hence suppose that
\[
  a_0+\sum (\A_2) \equiv 0 \pmod q.
\]
Now $|(\B \backslash \A_1) \bmod p|>4 \lceil 2 \sqrt{p} \rceil$,
and $\B \backslash \A_1$
contains at least one element not divisible by $m=p^{\ell-1}$. Hence, by Theorem
\ref{Ruzsa}, there exists $\A_3 \subset \B \backslash \A_1$ such that
\[
  \sum (\A_3) \equiv 0 \pmod p
\]
and
\[
  \sum (\A_3) \not \equiv 0 \pmod q.
\]
Set $\A=\A_2 \cup \A_3$ and note that $\A_2$ and $\A_3$ are disjoint. Hence
\[
  a_0+\sum(\A) = a_0+\sum(\A_2) + \sum(\A_3) \equiv 0+0 = 0 \pmod p
\]
and
\[
  a_0+\sum(\A) =
  a_0+\sum(\A_2) + \sum(\A_3) \equiv 0 + \sum (\A_3) \not \equiv 0
  \pmod q
\]
as required.
\end{proof}
\subsection{Application of the larger sieve}

\begin{proof}[Proof of Theorem \ref{densitytheorem}]


By Corollary \ref{schwarzwald}
we know that, for $p \in {\cal T}$, ${\cal A}$ lies in at most
$5 \lceil 2 \sqrt{p} \rceil+2$
distinct residue classes modulo $p$, for otherwise there exists an
element of $\mathcal{H}(a_0; a_1, \ldots, a_d)$ which is divisible by
$p$ but not by $p^r$.
This enables us to use Gallagher's larger sieve (Lemma \ref{lem:Gallagher}):

With $\nu(p) \le 5 \lceil 2\sqrt{p} \rceil +1$
the set $\mathcal{P}'$ in Lemma \ref{lem:Gallagher} is the set of 
primes in
$\mathcal{T}$
with $p \leq y =C_2 (\log N)^2$, 
where for example $C_2=(\frac{20}{\tau})^2$.

This gives the bound
\[{\cal A}(N) \leq \frac{- \log N + \sum_{p \in {\cal P}'} \log p}{
  - \log N +\sum_{p \in {\cal P}'} \frac{\log p}{5 \lceil 2\sqrt{p}
  \rceil+1}}
\leq \frac{2y}{-\log N+ \sqrt{y}\tau/10}\ll 
 \frac{\log N}{\tau^2},\]
which proves Theorem \ref{densitytheorem}.
\end{proof}

\begin{proof}[Proof of Theorem \ref{2nddensitytheorem}]
The proof is similar to that of Theorem \ref{densitytheorem}.
Let $p \not\in \cal{T}$.
If $|{\cal A}\bmod p| >2 \sqrt p$,
then Lemma \ref{olson} guarantees that all residue classes,
including the class $0$,
is contained in ${\cal H}\bmod p$ as a nonempty sum of the Hilbert cube.
As the Hilbert cube does not contain a sum being $0 \bmod p$ for any prime
$p \not\in {\cal T}$, then for all $p \not\in {\cal T}$
we have  $|{\cal A}\bmod p| \leq 2 \sqrt p$.
Then an application of Gallagher's larger sieve shows that for a
suitable $c=c(\tau)>0$ we have
\[
  |{\cal A}|\leq \frac{-\log N+\sum_{p \le y, p \not\in {\cal T}} \log p }{
 - \log N + \sum_{p \le y, p \not\in {\cal T}} \frac{\log p}{2 \sqrt{p}} }
  \leq \frac{2y}{- \log N + c\sqrt{y/2}}\ll_{\tau} \log N,
\]
with sufficiently large $y\gg_{\tau} (\log N)^2$.
\end{proof}

\ \\
\ \\
\section{Subset sums in pure powers, proof of Theorem \ref{thm:pure-powers}}
In two earlier papers \cite{DietmannandElsholtz, DietmannandElsholtz2}
the first two authors studied bounds on the dimension of subset sums or Hilbert cubes in various sets. For subsetsums in pure powers an important ingredient was an upper bound (\cite{DietmannandElsholtz2}, Lemma 4.8)
$ l \ll \frac{\log \log N}{\log \log \log N}$
on the maximal length $\ell$ of a homogeneous arithmetic progression in the set of pure powers. Taking advantage of recent
improvements on the sunflower lemma we can update our
earlier result to Theorem \ref{thm:pure-powers}.


To keep the proof short we comment on the required changes in the proof of Theorem 1.7 of \cite{DietmannandElsholtz2}.
Lemma 4.4 in that paper was actually proved in the precursor paper \cite{DietmannandElsholtz} Lemma 5 with an ad hoc inductive proof, 
but it is indeed a consequence
of the Erd\H{o}s-Rado sun flower theorem, see \cite{ErdosandRado:1960}.
This sunflower result has been recently improved, first by Alweiss et.~al.~\cite{Alweiss-Lovett-Wu-Zhang:2021}, and then by Bell et.~al.~\cite{Bell-Chueluecha-Warnke:2021}.
They proved: Let $v\geq 3$. For some positive constant $C$ and for any family ${\cal F}$
of sets with at most $h$ elements each the following holds:
If $|{\cal F}|\geq (v \log h)^h$, then there is a sunflower with $v$ \emph{petals}.
Applied to this situation this means:
There are $v$ sets for which the pairwise intersection equals the intersection of all $v$ sets, the \emph{kernel}.
Discarding the kernel of each of the $v$ sets there are 
$v$ disjoint sets of integers with identical sums $s$. This means that the homogeneous arithmetic progression
$0,s,2s, \ldots, (v-1)s$ of length $v$ is contained in the subset sum cube.

A combined and new version of Lemmas 4.4 and 4.5
of \cite{DietmannandElsholtz2} yields\footnote{The expression $v \geq 3$ corrects a misprint in the published version 4.4 of  
\cite{DietmannandElsholtz2}}:
\begin{lemma}
Let $A=\{a_1, \ldots, a_d\} \subset [1,N]$ be a set of distinct integers and let $H(0;a_1, \ldots, a_d)\subset S \cap [1,N]$ be the subset sum cube, where $S$ is a set of positive integers without a homogeneous arithmetic progression of length $v\geq 3$. 
Let $g(h,N)$ denote the maximum of the number of representations to write any integer in $[1,N]$ as a sum of $h$ summands of $A$, irrespective of 
the order of the elements.
Then, for any $h \leq d$,
$g(h,N)\leq (C v \log h)^h$.
Further, suppose that $d \geq 5h+4$,  and recall that $f(N)\leq (\log N)^{M}$, for some constant $M$.
Then 
$d\leq (5h!f(N)g(h,N))^{1/h}+5h+4$. 
\end{lemma}
Applying this to the set $V$ of pure powers gives:
\[d\leq (5h!)^{1/h} (C v \log h) e^{M(\log \log N)/h}+O(h).\] 
With $h=\lfloor \log \log N\rfloor$ 
and
$v=O(\frac{\log\log N}{\log \log \log N})$
this gives $d=O((\log \log N)^2)$.
\section{Proofs of Corollary \ref{cor:sumsof2squares} and Theorem \ref{2nddensitytheorem}}

Corollary \ref{cor:sumsof2squares} follows from Theorem \ref{densitytheorem}
if one recalls some properties of the values of binary quadratic forms.
Let $f(x,y)=a x^2+bxy+cy^2$  be an irreducible binary quadratic form with $a,b,c,x,y \in \Z$.
Let $d=b^2-4ac$ be the discriminant.


Let $p$ be a prime with $\left( \frac{b^2-4ac}{p} \right)=-1$,
and suppose that $p \mid n$ and $n$ is being represented by $f$,
i.e. $n=ax^2+bxy+cy^2$ for certain integers $x,y$.
Then $ax^2+bxy+cy^2 \equiv 0 \pmod p$. If $p \nmid y$, then
writing $w=x \overline{y}$, $\overline{y}$ being the inverse
of $y$ modulo $p$, we find that $aw^2+bw+c \equiv 0 \pmod p$,
whence the discriminant $d=b^2-4ac$ of the quadratic polynomial
$aw^2+bw+c$ must be a square modulo $p$, which contradicts
$\left( \frac{b^2-4ac}{p}\right)=-1$. As the same holds for
$p \nmid x$ we are left with the case $p \mid x$ and $p \mid y$. Here we have
$p^2 \mid n$.
 In other words, for the value set of this binary quadratic form $f$ there is a set of asymptotically about half of the primes such that
$p|f(x,y)$ implies $p^2|f(x,y)$.
The result then follows from theorem{\ref{densitytheorem}}.





\end{document}
\noindent Christian Elsholtz\\

TU GRAZ

\noindent Imre Z. Ruzsa\\
Alfr\'ed R\'enyi Institute of Mathematics\\
     Budapest, Pf. 127\\
     H-1364 Hungary
{\tt{ruzsa@renyi.hu}}
\end{document}